\documentclass[a4paper,10pt]{amsart}

\usepackage{amsmath, amssymb}
\usepackage[all]{xy}
\usepackage{color}
\usepackage{bm}

\newtheorem{thm}{Theorem}[section]
\newtheorem{prop}[thm]{Proposition}
\newtheorem{lem}[thm]{Lemma}
\newtheorem{cor}[thm]{Corollary}

\newtheorem*{thmA}{Theorem A}

\theoremstyle{definition}
\newtheorem{defn}[thm]{Definition}
\newtheorem{ex}[thm]{Example}

\theoremstyle{remark}
\newtheorem{remark}[thm]{Remark}

\newcommand{\map}{\operatorname{map}}
\newcommand{\aut}{\operatorname{aut}}
\newcommand{\Der}{\operatorname{Der}}
\newcommand{\Hom}{\operatorname{Hom}}
\newcommand{\ev}{\operatorname{ev}}
\newcommand{\ad}{\operatorname{ad}}
\newcommand{\id}{\operatorname{id}}
\newcommand{\mpair}[2]{\langle \hspace{-0.2em} \langle #1, #2 \rangle \hspace{-0.2em} \rangle}

\numberwithin{equation}{section}

\title[Higher order Whitehead products and defining systems]
{Rational higher order Whitehead products of mapping spaces and defining systems}
\author{
Takahito Naito
}
\date{}
\address{
Nippon Institute of Technology,
4-1 Gakuendai, Miyashiro-machi, Minamisaitama-gun, Saitama, 345-8501 Japan
}
\email{naito.takahito@nit.ac.jp}

\begin{document}

\begin{abstract}
This paper gives an algebraic description of rational higher order Whitehead products in mapping spaces.
Given a Sullivan representative of the base map, we introduce defining systems in the corresponding derivation complex.
The main result shows that, under the standard identification between rational homotopy groups of
mapping spaces and the cohomology of the derivation complex, rational higher order Whitehead products are given exactly by the cohomology classes represented by the associated derivation brackets. 
The examples show how choices of defining systems produce non-trivial elements and indeterminacy.
\end{abstract}

\maketitle

\section{Introduction}

Higher order Whitehead products were introduced by Porter \cite{Po1965} as higher
homotopy operations which generalize the ordinary Whitehead product. For
$p>2$, a $p$-th order product is non-empty precisely when all lower order
Whitehead products associated with proper subcollections vanish.
These products appear naturally in the study of attaching maps, for instance in
toric topology and polyhedral products \cite{GT2010,IK2020}.
This motivates the question of how higher order Whitehead products can be described in rational homotopy theory.

Andrews and Arkowitz studied this problem using minimal Sullivan models
\cite{AA1978}. Their work showed that higher order Whitehead products can be detected
from the higher order terms in the differential of a minimal Sullivan model.
This result extends the standard Sullivan-model description of the ordinary
Whitehead product by the quadratic part of the differential
\cite[\S 13(e)]{FHT}, and gives an algebraic description of higher order
Whitehead products in terms of rational models.

The rational homotopy theory of mapping spaces has also been studied using
Sullivan models.
In particular, Brown and Szczarba \cite{BS1997} constructed a model for
mapping spaces.
Hirato, Kuribayashi and Oda \cite{HKO2008} used Brown--Szczarba
models to study rational evaluation subgroups, and related non-Gottlieb elements
to higher order Whitehead products.
Buijs and Murillo  \cite{BM2008} described the rational homotopy Lie algebra of mapping spaces in terms of derivation complexes and gave derivation brackets detecting Whitehead products and higher order
Whitehead products.
Later, Buijs, F{\'e}lix and Murillo \cite{BFM2012} constructed explicit $L_\infty$ models for components of free and pointed
mapping spaces. There are also related results of Belch{\'i},
Buijs, Moreno-Fern{\'a}ndez and Murillo \cite{BBMM2017} on higher order Whitehead products
and higher brackets in transferred $L_\infty$ structures.

In this paper, defining systems in a derivation complex are used to study higher order Whitehead products in mapping spaces. The starting point is the definition of higher order Whitehead products in terms of extensions over the fat wedge. 
More precisely, a Sullivan representative of such an extension is written with respect to the natural basis of the cohomology of the fat wedge.
Its coefficients then give derivations satisfying the differential identities
required for a defining system.

To obtain these identities, the Sullivan representative is first replaced by a homotopic strict representative. 
This strict representative allows us to compare the commutative differential graded algebra (CDGA) condition with the brackets in the derivation complex.
As a consequence, the associated bracket of the defining system represents the corresponding rational
higher order Whitehead product in the mapping space.
This result extends the author's previous description of the ordinary Whitehead product to higher order
Whitehead products \cite{Na2014}.

Throughout the paper, we use the standard terminology and notation of rational homotopy theory, including CDGAs, Sullivan models, and Sullivan representatives,
as in F{\'e}lix--Halperin--Thomas \cite{FHT}. The notation
$\map(X,Y;f)$ denotes the path component of the mapping space $\map(X,Y)$
containing $f:X\to Y$, with basepoint $f$.
For a map $f:X\to Y$, let $(\wedge V,d)$ be a minimal Sullivan model for
$Y$, $(A,d_A)$ a CDGA model for $X$, and
$
        \phi:(\wedge V,d)\to(A,d_A)
$
a Sullivan representative for $f$.
We write $\Der(\wedge V,A;\phi)$ for the cochain complex of $\phi$-derivations; see Section~4 for its definition.

Before stating the main theorem, we recall the standard identification of the
rational homotopy groups of mapping spaces with the cohomology of a derivation
complex. By Lupton and Smith~\cite[Theorem~2.1]{LS2007}, this identification is
given, in our grading, by
\[
        \Phi:\pi_n(\map(X,Y;f))\otimes\mathbb{Q}
        \cong
        H^{-n}(\Der(\wedge V,A;\phi)).
\]
Here rational higher order Whitehead products are understood in the
rationalization of the mapping-space component; see Section~5 for details.

\begin{thmA}
Let $X$ be a simply connected finite CW complex, $Y$ a simply connected
space of finite type. Let $(\wedge V,d)$ be a minimal Sullivan model for $Y$,
$(A,d_A)$ a CDGA model for $X$, and
$
        \phi:(\wedge V,d)\to (A,d_A)
$
a Sullivan representative for a map $f:X\to Y$.
Let $p\geq2$. Suppose that integers $n_i\geq2$, classes
$
        \alpha_i\in
        \pi_{n_i}(\map(X,Y;f))\otimes\mathbb Q,
$
and cycles
$
        \theta_i\in\Der^{-n_i}(\wedge V,A;\phi)
$
representing $\Phi(\alpha_i)$ are given for $i=1,\ldots,p$.

If the higher order Whitehead product
\[
        [\alpha_1,\ldots,\alpha_p]_{\mathrm{Wh}}
\]
is non-empty, then its image under $\Phi$ is exactly the set of cohomology
classes represented by
\[
        [\theta_1,\ldots,\theta_p]_\Theta,
\]
where $\Theta$ runs over all defining systems for
$\theta_1,\ldots,\theta_p$.
\end{thmA}

The precise definition of a defining system and of the associated bracket
$[\theta_1,\ldots,\theta_p]_\Theta$ is given in
Definition~\ref{def:defining-system}. Theorem~A is proved as
Theorem~\ref{thm:main}.

The relation between Theorem~A and the results of Andrews and Arkowitz \cite{AA1978} and Buijs and Murillo \cite{BM2008} is discussed after Theorem~\ref{thm:main}.
When $X=*$, Theorem~A recovers the formula of Andrews and Arkowitz.
For general mapping spaces, defining systems combine the brackets of Buijs and Murillo and describe the resulting indeterminacy.
The examples in Section~6 illustrate the theorem. They show a non-zero
third-order product with a unique defining system, dependence of defining
systems and indeterminacy on the base map, and a non-zero associated bracket
arising from higher components even when the Sullivan model has only quadratic
differentials.

This paper is organized as follows. 
Section~2 recalls higher order Whitehead products and the universal Whitehead product associated with the fat wedge.
Section~3 studies Sullivan models and cohomology models for the fat wedge and gives a Sullivan representative for the universal higher order Whitehead product.
Section~4 studies strict representatives and the differential identities satisfied by their coefficients.
Section~5 introduces defining systems in the derivation complex and proves the main theorem for mapping
spaces.
Section~6 gives examples illustrating the theorem and the indeterminacy arising from choices of defining systems.

\section{Higher order Whitehead products}

In this section, we recall Porter's definition of higher order Whitehead products
and fix the notation used throughout the paper. All spaces are assumed to be
pointed, and all maps and homotopies preserve basepoints. Unless otherwise
specified, homology groups in this section have integer coefficients. When no
confusion can arise, we do not distinguish between a map and its homotopy class.

Let $S^{n_1}, \ldots, S^{n_p}$ be spheres with $p\geq2$ and $n_i\geq 2$.
Put
\[
P=S^{n_1}\times\cdots\times S^{n_p},
\quad
W=S^{n_1}\vee\cdots\vee S^{n_p}
\quad
\text{and}
\quad
T=T(S^{n_1},\ldots,S^{n_p})
\]
for the Cartesian product, the wedge sum and the fat wedge, respectively.
Thus $T$ is the subspace of $P$ consisting of elements with at least one
coordinate equal to the basepoint. We write
\[
N=n_1+\cdots+n_p.
\]

The product $P$ is obtained from the fat wedge $T$ by attaching one
$N$-cell:
\[
P=T\cup_w e^N.
\]
Thus the relative CW complex $(P,T)$ has a single cell in dimension $N$.
In particular, the pair $(P,T)$ is $(N-1)$-connected.
Hence the relative Hurewicz homomorphism $\operatorname{Hur} : \pi_N(P,T) \to H_N(P,T)$ is an isomorphism.
The attaching map
$
w:S^{N-1} \to T
$
is called the \emph{universal} $p$-th order Whitehead product. More precisely, the
homotopy class of $w$ is given by the image of the fundamental class
$
[P]\in H_N(P)
$
under the composite
\begin{equation}\label{eq:universal-HOP}
H_N(P)
\xrightarrow{j_*}
H_N(P,T)
\xrightarrow{\operatorname{Hur}^{-1}}
\pi_N(P,T)
\xrightarrow{\partial}
\pi_{N-1}(T),
\end{equation}
where $j_*$ is induced by the inclusion of pairs, $\operatorname{Hur}$ is the
relative Hurewicz isomorphism, and $\partial$ is the connecting homomorphism.
Here, under the K{\"u}nneth isomorphism, the class $[P]$ is identified with
$
[S^{n_1}]\otimes\cdots\otimes[S^{n_p}]
\in
H_{n_1}(S^{n_1})\otimes \cdots \otimes H_{n_p}(S^{n_p}).
$

Let $X$ be a pointed space and let
$
\alpha_i\in \pi_{n_i}(X)
$
be represented by $g_i:S^{n_i}\to X$ for $i=1,\ldots,p$.
The maps $g_i$ determine a map
\[
g:W=S^{n_1}\vee\cdots\vee S^{n_p}
\longrightarrow X.
\]
The \emph{$p$-th order Whitehead product} of
$\alpha_1,\ldots,\alpha_p$ is defined as the subset
\[
[\alpha_1,\ldots,\alpha_p]_{\mathrm{Wh}}
\subset
\pi_{N-1}(X)
\]
given by
\[
[\alpha_1,\ldots,\alpha_p]_{\mathrm{Wh}}
=
\left\{
\widetilde g\circ w
\ \middle|\
\widetilde g:T\to X
\text{ is an extension of }g
\right\}.
\]
Thus this set is non-empty precisely when the map $g:W\to X$ extends to the
fat wedge $T$. Note that, for $p=2$, the fat wedge is the wedge sum
$
T(S^{n_1},S^{n_2})=S^{n_1}\vee S^{n_2},
$
and the above construction recovers the ordinary Whitehead product.

We shall use the following existence criterion due to Porter.

\begin{thm}[{\cite[Theorem 2.7]{Po1965}}]\label{thm:porter}
For $p>2$, the higher order Whitehead product
\[
[\alpha_1,\ldots,\alpha_p]_{\mathrm{Wh}}
\]
is non-empty if and only if all lower order Whitehead products formed by
proper subcollections of $\{\alpha_1,\ldots,\alpha_p\}$ vanish.
Here, a set-valued Whitehead product is said to vanish if it contains the
zero element.
\end{thm}

In the sequel, we mainly use the above description of
$[\alpha_1,\ldots,\alpha_p]_{\mathrm{Wh}}$
in terms of extensions over the fat wedge. In Sections 4 and 5, we translate
such extensions into algebraic data in derivation complexes.

\section{Sullivan models and cohomology models for the fat wedge}

From this section on, all CDGAs are taken over the rational number field
$\mathbb{Q}$. 
We recall Sullivan models for $P$ and $T$, together with the cohomology
models used below.
The Sullivan model for the fat wedge was studied by Andrews and Arkowitz
\cite[Section~4]{AA1978}. The construction is recalled from the viewpoint of
relative Sullivan models.

For each $i$, let
\[
m_i:\mathcal M_{S^{n_i}}=(\wedge V_i,d)
\longrightarrow A_{\mathrm{PL}}(S^{n_i})
\]
be the minimal Sullivan model for $S^{n_i}$, where
\[
(\wedge V_i,d)
=
\begin{cases}
(\wedge(e_i),d=0), & n_i \text{ odd},\\[2mm]
(\wedge(e_i,e'_i),de_i=0,\ de'_i=e_i^2), & n_i \text{ even}.
\end{cases}
\]
Here $|e_i|=n_i$ and, in the even-dimensional case,
$|e'_i|=2n_i-1$; see \cite[\S 12(a), Example~1]{FHT}.
The minimal Sullivan model for the product $P$ is given by
\[
\mathcal M_P
=
\bigotimes_{i=1}^p \mathcal M_{S^{n_i}}
\cong
\bigl(\wedge(V_1\oplus\cdots\oplus V_p),d\bigr),
\]
equipped with a quasi-isomorphism
$m_P:\mathcal M_P\to A_{\mathrm{PL}}(P)$; see
\cite[\S 12(a), Example~2]{FHT}.

Let $j_T:T\hookrightarrow P$ be the canonical inclusion.
By the relative Sullivan model construction
\cite[Proposition~14.3]{FHT}, the CDGA morphism
$
j_T^*\circ m_P:\mathcal M_P\to A_{\mathrm{PL}}(T)
$
admits a relative Sullivan model
\[
\mathcal M_P
\longrightarrow
\mathcal M_T
=
\bigl(\wedge(V_1\oplus\cdots\oplus V_p)\otimes\wedge V_F,d\bigr)
\xrightarrow{\simeq}
A_{\mathrm{PL}}(T).
\]
We may choose $\mathcal M_T$ to be a minimal Sullivan model of $T$.
Set
\[
V_P=V_1\oplus\cdots\oplus V_p
\quad
\text{and}
\quad
V_T=V_P\oplus V_F.
\]

We now describe the part of $\mathcal M_T$ which is relevant for the universal
higher order Whitehead product. Put
\[
N=n_1+\cdots+n_p,
\quad
\text{and}
\quad
e_{[p]}=e_1e_2\cdots e_p\in \mathcal M_P^N.
\]
The inclusion $j_T:T\hookrightarrow P$ induces an isomorphism
$
j_T^*:H^q(P;\mathbb{Q})\xrightarrow{\cong} H^q(T;\mathbb{Q})
$
 for $q<N$,
while the top class represented by $e_{[p]}$ maps to zero in
$H^N(T;\mathbb{Q})$. Thus, in the relative Sullivan model for $T$, we may
choose $V_F$ so that
\[
V_F=V_F^{\geq N-1},
\quad
\text{and}
\quad
\dim V_F^{N-1}=1.
\]
Let $u_{[p]}\in V_F^{N-1}$ be a generator. We may assume that
$
du_{[p]}=e_{[p]}.
$

Since $P$, $T$ and $W$ are formal, we use their rational cohomology
algebras, equipped with the zero differential, as CDGA models. Set
\[
\mathcal H_P=H^*(P;\mathbb Q),\qquad
\mathcal H_T=H^*(T;\mathbb Q),\qquad
\mathcal H_W=H^*(W;\mathbb Q).
\]
Let
$
\rho:\mathcal H_T \to \mathcal H_W
$
be the morphism induced by the inclusion $W\hookrightarrow T$.
We choose the formality quasi-isomorphisms compatibly with this inclusion,
taking
$q_T :  \mathcal{M}_T \to \mathcal{H}_T$
to send each $e_i$ to its cohomology class and all other generators to zero.
Then the composite
$ \rho\circ q_T : \mathcal{M}_T \to \mathcal{H}_W$
is a Sullivan representative for the inclusion $W\hookrightarrow T$.

For a subset
$
I=\{i_1<\cdots<i_k\}\subset [p]=\{1,\ldots,p\},
$
write
$
e_I=e_{i_1}\cdots e_{i_k}
$
and
$
|e_I|=\sum_{i\in I}n_i,
$
with the convention $e_\emptyset=1$.

The quotient $P/T$ is homeomorphic to the smash product
$
S^{n_1}\wedge\cdots\wedge S^{n_p}
\cong
S^N.
$
Hence the top class $e_{[p]}$ becomes zero in passing from $P$ to $T$, and
we have
$
\mathcal H_T\cong \mathcal H_P/(e_{[p]}).
$
Equivalently, $\mathcal H_T$ has the homogeneous basis
$
\{e_I\}_{I\subsetneq[p]}.
$
Similarly,
\[
\mathcal H_W
\cong
\mathbb{Q}\oplus\bigoplus_{i=1}^p \mathbb{Q} \langle e_i \rangle,
\]
and all products of positive-degree elements vanish in $\mathcal H_W$.
With respect to the above bases, $\rho$ is given by
$
\rho(e_i)=e_i
$
for $i=1,\ldots,p$, and
$
\rho(e_I)=0
$
for $|I|\geq 2$.

Let $F$ be the homotopy fiber of the inclusion $j_T:T\hookrightarrow P$.
A Sullivan model for $F$ is obtained by quotienting $\mathcal M_T$ by
the ideal generated by $V_P$, and is therefore identified with
\[
\mathcal M_F=(\wedge V_F,\bar d),
\]
where $\bar d$ is induced by $d$; see \cite[Theorem 15.3]{FHT}.
By construction, we have
\[
\mathcal M_F^{N-1}=V_F^{N-1}=\mathbb{Q}\langle u_{[p]}\rangle.
\]
Moreover, in $\mathcal M_T$, the differential satisfies
$du_{[p]}=e_{[p]}$.
Equivalently, the part of the differential relevant to the top class is the
linear map
\[
d:V_F^{N-1}\longrightarrow \mathcal M_P^N,
\qquad
u_{[p]}\longmapsto e_{[p]}.
\]

Recall the composite map
$
\Xi=\partial\circ \operatorname{Hur}^{-1}\circ j_*
$
defined in \eqref{eq:universal-HOP}. This map represents the universal
Whitehead product. After rationalization, it gives
\[
\Xi_{\mathbb{Q}}:H_N(P;\mathbb{Q})
\longrightarrow
\pi_{N-1}(T)\otimes\mathbb{Q}.
\]

We use the standard identification
\[
\pi_q(Z)\otimes \mathbb{Q}
\cong
\operatorname{Hom}(V_Z^q,\mathbb{Q})
\]
for a simply connected space $Z$ of finite type with minimal Sullivan model
$(\wedge V_Z,d)$, where $q\geq 2$; see \cite[Theorem 15.11]{FHT}.

The following lemma describes the linear dual of $\Xi_{\mathbb{Q}}$ in terms of
the above Sullivan models.

\begin{lem}\label{lem:dual_composite}
Under the standard identification
$
\pi_{N-1}(T)\otimes\mathbb{Q}
\cong
\operatorname{Hom}(V_T^{N-1},\mathbb{Q}),
$
the linear dual of
$
\Xi_{\mathbb{Q}}:H_N(P;\mathbb{Q})
\to
\pi_{N-1}(T)\otimes\mathbb{Q}
$
is represented by the composite
\[
\xymatrix{
V_T^{N-1}
\ar[r]^-{Q(\pi_F)}
&
V_F^{N-1}
\ar[r]^-{q^{-1}}_-{\cong}
&
\mathcal M_F^{N-1}
\ar[r]^-{d}
&
\mathcal M_P^N.
}
\]
Here $Q(\pi_F)$ is the linear part of the quotient map
$
\pi_F:\mathcal M_T\to\mathcal M_F,
$
and $q:\mathcal M_F\to V_F$ is the projection onto the indecomposables.
Moreover, the target $\mathcal M_P^N$ is regarded as representing
$H^N(P;\mathbb{Q})$ through the cohomology class of $e_{[p]}$.
\end{lem}

\begin{proof}
Let $i:F\to T$ be the canonical map
from the homotopy fiber. By the naturality of the Hurewicz homomorphisms and
the comparison between the long exact sequences of the pair $(P,T)$ and of the
fibration, the following diagram commutes:
\[
\xymatrix@C=35pt{
H_N(P) \ar[d]_-{\tau} \ar[r]^-{j_*}
&
H_N(P,T) \ar[r]^-{\operatorname{Hur}^{-1}}_-{\cong}
&
\pi_N(P,T) \ar[d]^-{\cong} \ar[r]^-{\partial}
&
\pi_{N-1}(T)
\\
H_{N-1}(F) \ar[rr]^-{\operatorname{Hur}^{-1}}_-{\cong}
&
&
\pi_{N-1}(F) \ar[ru]_-{i_*}
&
}
\]
Here $\tau$ is the homological transgression; see
\cite[Definition~6.3]{McC}. Hence
\[
        \Xi=i_*\circ\operatorname{Hur}^{-1}\circ\tau .
\]

Taking the linear duals of the three maps in this factorization, and using the
standard Sullivan identification of rational homotopy groups with the duals of
indecomposables, the projection
$
q:\mathcal M_F\to V_F
$
corresponds to the linear dual of the Hurewicz homomorphism
$
\operatorname{Hur}:\pi_{N-1}(F)\otimes\mathbb{Q}
\to
H_{N-1}(F;\mathbb{Q})
$;
see \cite[\S 15]{FHT}. Since
$
\pi_F:\mathcal M_T\to\mathcal M_F
$
is a Sullivan representative for $i:F\to T$, its linear part $Q(\pi_F)$ is the
linear dual of $i_*$. Finally, the differential
$
d:V_F^{N-1}\to \mathcal M_P^N
$
represents the cohomological transgression, which is dual to $\tau$.

Therefore the linear dual of
$
\Xi=i_*\circ\operatorname{Hur}^{-1}\circ\tau
$
is represented by
$
d\circ q^{-1}\circ Q(\pi_F).
$
This completes the proof.
\end{proof}

Recall that Andrews and Arkowitz \cite[Proposition~4.7]{AA1978}
computed the universal $p$-th order Whitehead product using the Sullivan
pairing. We rewrite their result in terms of the relative Sullivan model.

\begin{prop}\label{prop:model_universal-HOP}
Let
$
\omega:\mathcal M_T\to \mathcal M_{S^{N-1}}
$
be a Sullivan representative for the universal $p$-th order Whitehead product
$
w:S^{N-1}\to T.
$
Then
\[
        \omega(u_{[p]})=(-1)^{N-1+\kappa}e,
\]
where $e\in \mathcal M_{S^{N-1}}^{N-1}$ is the fundamental generator and
\[
        \kappa=\sum_{1\leq i<j\leq p} n_i n_j .
\]
Moreover, the linear part $Q(\omega)$ vanishes on $V_P^{N-1}\subset V_T^{N-1}$.
\end{prop}

\begin{proof}
Let
$
Q(\omega):V_T^{N-1}\to \mathbb Q\langle e\rangle\cong \mathbb Q
$
be the linear part of $\omega$ in degree $N-1$. We write
\[
        Q(\omega)(u_{[p]})=\lambda e
\]
for some $\lambda\in\mathbb Q$. It is enough to prove that
$
\lambda=(-1)^{N-1+\kappa}.
$

Consider the following commutative diagram:
\[
\xymatrix@C=12pt{
V_{T}^{N-1} \times \mathrm{Hom}(V_{T}^{N-1} , \mathbb{Q})
\ar[rr]^-{\mpair{ \ }{ \ }_T} && \mathbb{Q}
\\
V_{T}^{N-1} \times \mathrm{Hom}(V_{F}^{N-1} , \mathbb{Q}) \ar[u]^-{1 \times Q(\pi_F)^*} &&
\\
V_{T}^{N-1} \times \mathrm{Hom}(\mathcal{M}_{F}^{N-1} , \mathbb{Q}) \ar[u]_-{\cong}^-{1\times (q^{-1})^*} \ar[r]^-{\pi_F \times 1} & \mathcal{M}_{F}^{N-1} \times \mathrm{Hom}(\mathcal{M}_{F}^{N-1}, \mathbb{Q}) \ar[ruu]^-{\mpair{ \ }{ \ }_{F}} &
\\
V_{T}^{N-1} \times \mathrm{Hom}(\mathcal{M}_{P}^{N} , \mathbb{Q}) \ar[u]^-{1 \times d^*} 
\ar[r]^-{\pi_F \times 1}
& \mathcal{M}_{F}^{N-1} \times \mathrm{Hom}(\mathcal{M}_{P}^{N} , \mathbb{Q})  \ar[r]^-{d \times 1} \ar[u]^-{1\times d^*} & \mathcal{M}_{P}^{N} \times \mathrm{Hom}(\mathcal{M}_{P}^{N} , \mathbb{Q}). \ar[uuu]^-{\mpair{ \ }{ \ }_{P}}
}
\]
Here, $\mpair{\ }{\ }_T$, $\mpair{\ }{\ }_F$ and $\mpair{\ }{\ }_P$ are the evaluation maps.
The maps $\mpair{\ }{\ }_F$ and $\mpair{\ }{\ }_P$ correspond to the Kronecker pairings between homology and cohomology.

Let $e_{[p]}^*\in \operatorname{Hom}(\mathcal M_P^N,\mathbb Q)$ be the dual
element of
$
e_{[p]}=e_1e_2\cdots e_p,
$
so that $e_{[p]}^*(e_{[p]})=1$. Under the Kronecker pairing, the fundamental
class $[P]\in H_N(P;\mathbb Q)$ is represented by
$
(-1)^{\kappa} e_{[p]}^*.
$
This sign comes from the Koszul sign rule for the pairing between the product
of cohomology classes and the cross product of homology classes; see
\cite[VI.4.12 Example]{Bre}.

Lemma~\ref{lem:dual_composite} shows that the composite of the vertical maps
on the left-hand side is the linear dual of $\Xi_{\mathbb Q}$, up to the
above identifications. Hence, by the definition of the universal $p$-th order
Whitehead product,
\[
        Q(\omega)
        =
        (-1)^{\kappa}
        Q(\pi_F)^*\circ (q^{-1})^*\circ d^*(e_{[p]}^*).
\]
By the commutativity of the diagram, we have
\begin{align*}
\lambda
&= \mpair{u_{[p]}}{Q(\omega)}_T                                      \\
&= (-1)^\kappa
\mpair{u_{[p]}}
{Q(\pi_F)^*\circ (q^{-1})^*\circ d^*(e_{[p]}^*)}_T                    \\
&= (-1)^{N-1+\kappa}
\mpair{du_{[p]}}{e_{[p]}^*}_P                                         \\
&= (-1)^{N-1+\kappa}
\mpair{e_{[p]}}{e_{[p]}^*}_P                                          \\
&= (-1)^{N-1+\kappa}.
\end{align*}
Here the sign $(-1)^{N-1}$ comes from taking the dual of the differential of degree $+1$. 

Moreover, let $\iota : \mathcal{M}_P \hookrightarrow \mathcal{M}_T$ be the inclusion which is a Sullivan representative for $j_T : T\hookrightarrow P$. 
The composite
\[
\xymatrix{
\mathcal{M}_P   \ar[r]^-{\iota}
&
\mathcal{M}_T \ar[r]^-{\omega}
&
\mathcal{M}_{S^{N-1}}
}
\]
is a Sullivan representative for $j_T\circ w$.
As $w$ is the attaching map of the top cell in $P=T\cup_w e^N$, the composite $j_T\circ w$ is null-homotopic, so the linear part of $\omega\circ\iota$ is zero.
Since $Q(\iota)$ is the inclusion $V_P\hookrightarrow V_T$, it follows that $Q(\omega)|_{V_P^{N-1}}=0$.
This completes the proof.
\end{proof}

\begin{remark}\label{rem:sign}
In \cite[Proposition~4.7]{AA1978}, Andrews and Arkowitz obtain the sign
$(-1)^\kappa$ using the Sullivan pairing. The additional factor
$(-1)^{N-1}$ in Proposition~\ref{prop:model_universal-HOP} comes from our
convention of using CDGA representatives and taking the dual of the
differential of degree $+1$.
\end{remark}

In the next section, we use the cohomology model $\mathcal H_T$ and the map
$\rho:\mathcal H_T\to\mathcal H_W$ to obtain algebraic defining systems from
CDGA representatives over the fat wedge.


\section{Differential identities and strict representatives}

In this section, we prepare the algebraic construction that will be used later
for higher order Whitehead products in mapping spaces. Throughout this section, let
$(\wedge V,d)$ be a Sullivan algebra, $(A,d_A)$ a CDGA, and
$
\phi:\wedge V\to A
$
a CDGA morphism.

A \emph{$\phi$-derivation} of degree $r$
is a linear map $\theta:\wedge V\to A$ of degree $r$ satisfying
\[
        \theta(ab)
        =
        \theta(a)\phi(b)+(-1)^{r|a|}\phi(a)\theta(b)
\]
for $a,b\in \wedge V$.
Denote by $\Der^r(\wedge V,A;\phi)$ the vector space of
$\phi$-derivations of degree $r$. These vector spaces form a cochain complex
$\Der(\wedge V,A;\phi)$ with differential
\[
        \delta\theta
        =
        d_A\theta-(-1)^r \theta d
\]
for $\theta\in \Der^r (\wedge V,A;\phi)$.
Since $\wedge V$ is free as a graded commutative algebra, a $\phi$-derivation
is determined by its restriction to $V$.

\subsection{Buijs--Murillo brackets}

In this subsection, we recall the bracket on derivations introduced by Buijs and Murillo \cite{BM2008}.
Accordingly, a homogeneous linear map $V\to A$ will be regarded as the corresponding $\phi$-derivation on $\wedge V$ when no confusion can arise.
Let $\theta_i:V\to A$ be homogeneous linear maps for $i=1,\ldots,k$.
Define a linear map
\[
        \langle \theta_1,\ldots,\theta_k\rangle_{\phi}
        :\wedge V\longrightarrow A
\]
as follows. For homogeneous elements $v_1,\ldots,v_m\in V$, set
\[
        \langle \theta_1,\ldots,\theta_k\rangle_{\phi}
        (v_1\cdots v_m)=0
\]
if $m<k$. If $m\geq k$, then
\begin{align*}
&\langle \theta_1,\ldots,\theta_k\rangle_{\phi}(v_1\cdots v_m) \\
&=
\begin{cases}
\displaystyle
\sum_{\sigma\in \mathfrak S_k}
(-1)^{\varepsilon(\theta,v,\sigma)}
\theta_1(v_{\sigma(1)})\cdots \theta_k(v_{\sigma(k)})
& (m=k), \\[3mm]
\displaystyle
\sum_{\tau\in \mathfrak S_{k,m-k}}
(-1)^{\varepsilon(v,\tau)}
\langle \theta_1,\ldots,\theta_k\rangle_{\phi}
(v_{\tau(1)}\cdots v_{\tau(k)})
\phi(v_{\tau(k+1)}\cdots v_{\tau(m)})
& (m>k).
\end{cases}
\end{align*}
Here $\mathfrak S_k$ denotes the symmetric group on $k$ letters, and
$\mathfrak S_{k,m-k}$ denotes the set of $(k,m-k)$-shuffles.
The sign $(-1)^{\varepsilon(v,\tau)}$ is the Koszul sign determined by
\[
        v_{\tau(1)}\cdots v_{\tau(m)}
        =
        (-1)^{\varepsilon(v,\tau)}
        v_1\cdots v_m .
\]
In the case $m=k$, the
sign is determined by
\[
\varepsilon(\theta,v,\sigma)
=
\varepsilon(v,\sigma)
+
\sum_{i=2}^k |\theta_i|
\left(\sum_{j=1}^{i-1}|v_{\sigma(j)}|\right).
\]
When $k=1$, the map
$
\langle\theta_1\rangle_{\phi}:\wedge V\to A
$
is just $\theta_1$, regarded as the corresponding $\phi$-derivation.

For $k\geq 2$, the Buijs--Murillo bracket
$
[\theta_1,\ldots,\theta_k]\in \Der(\wedge V,A;\phi)
$
is defined by
\[
[\theta_1,\ldots,\theta_k](v)
=
-
(-1)^{|\theta_1|+\cdots+|\theta_k|}
\langle \theta_1,\ldots,\theta_k\rangle_{\phi}(dv)
\]
for $v\in V$. 
For $k=1$, we define the bracket to be the differential of $\theta_1$:
\[
        [\theta_1]=\delta\theta_1
        =
        d_A\theta_1-(-1)^{|\theta_1|}\theta_1d.
\]

\subsection{Component decomposition}

Let $\mathcal H_T=H^*(T;\mathbb{Q})$ be the rational cohomology algebra of the
fat wedge $T$, as described in Section 3.
In this subsection, we
decompose a CDGA morphism
$
\varphi:\wedge V\to \mathcal H_T\otimes A
$
with respect to the basis of $\mathcal H_T$.
Recall that $\mathcal H_T$ has the basis
$\{e_I\}_{I\subsetneq [p]}$.
Hence we can write
\[
        \varphi=\sum_{I\subsetneq [p]} e_I\otimes \varphi_I
\]
for homogeneous linear maps $\varphi_I:\wedge V\to A$ of degree $-|e_I|$.
The component $\varphi_{\emptyset}$ is a CDGA morphism, since it is obtained
by composing $\varphi$ with the augmentation
$
\mathcal H_T\to \mathbb Q.
$
The other components satisfy the following identity.

\begin{prop}\label{prop:component-decomposition}
Let $I\subsetneq [p]$ be non-empty. Then
\[
\varphi_I
=
\sum_{\{I_1,\ldots,I_k\}\in \Pi(I)}
(-1)^{\chi(I_1,\ldots,I_k)}
\langle \varphi_{I_1},\ldots,\varphi_{I_k}\rangle_{\varphi_{\emptyset}}.
\]
Here $\Pi(I)$ denotes the set of all partitions of $I$ into non-empty subsets.
For each partition, we order the subsets by
$
\min I_1<\cdots<\min I_k.
$
On the right-hand side, each $\varphi_{I_r}$ is regarded as its restriction
$\varphi_{I_r}|_V:V\to A$. 
The sign is given by
\[
        \chi(I_1,\ldots,I_k)
        =
        \sum_{q<r}|\varphi_{I_r}||\varphi_{I_q}|.
\]
\end{prop}

\begin{proof}
Let $v_1,\ldots,v_m\in V$ be homogeneous elements. We compare the
coefficient of $e_I$ in the identity
\[
        \varphi(v_1\cdots v_m)
        =
        \varphi(v_1)\cdots \varphi(v_m).
\]
The coefficient of $e_I$ on the left-hand side is
$\varphi_I(v_1\cdots v_m)$.

Let $\mathcal D_m(I)$ be the set of ordered $m$-tuples
$(J_1,\ldots,J_m)$ of subsets of $I$ such that
\[
        I=J_1\sqcup\cdots\sqcup J_m,
\]
where empty subsets are allowed. Then the terms in
$\varphi(v_1)\cdots\varphi(v_m)$ which contribute to the coefficient of
$e_I$ are
\begin{align*}
&\sum_{(J_1,\ldots,J_m)\in \mathcal D_m(I)}
(e_{J_1}\otimes \varphi_{J_1}(v_1))
\cdots
(e_{J_m}\otimes \varphi_{J_m}(v_m))                       \\
&=
e_I\otimes
\left(
\sum_{(J_1,\ldots,J_m)\in \mathcal D_m(I)}
(-1)^{\varepsilon(J_1,\ldots,J_m)+\eta}
\varphi_{J_1}(v_1)\cdots\varphi_{J_m}(v_m)
\right).
\end{align*}
Here the signs are defined by
\[
        e_{J_1}\cdots e_{J_m}
        =
        (-1)^{\varepsilon(J_1,\ldots,J_m)}e_I
\quad
\text{and}
\quad
        \eta
        =
        \sum_{q<r}|e_{J_r}|\,
        |\varphi_{J_q}(v_q)|.
\]
The second sign comes from moving the elements of $\mathcal H_T$ past the
elements of $A$ in the tensor product algebra $\mathcal H_T\otimes A$.

Now rewrite this sum by grouping together the non-empty subsets among
$J_1,\ldots,J_m$. Let $\Pi_k(I)$ be the set of partitions of $I$ into
$k$ non-empty subsets. For each partition
$\{I_1,\ldots,I_k\}\in \Pi_k(I)$, we use the order fixed in the statement,
namely
$
\min I_1<\cdots<\min I_k.
$
Then the above sum is equal to
\begin{align*}
&\sum_{k=1}^m
\sum_{\{I_1,\ldots,I_k\}\in \Pi_k(I)}
(-1)^{\chi(I_1,\ldots,I_k)}
\langle
\varphi_{I_1},\ldots,\varphi_{I_k}
\rangle_{\varphi_{\emptyset}}
(v_1\cdots v_m).
\end{align*}
Indeed, the terms with exactly these non-empty subsets are precisely the
terms appearing in the definition of
$
\langle
\varphi_{I_1},\ldots,\varphi_{I_k}
\rangle_{\varphi_{\emptyset}}.
$
The Koszul signs coming from the products
$e_{J_1}\cdots e_{J_m}$ are included in the sign of this multilinear map.
Using $|e_{J_r}|=-|\varphi_{J_r}|$, the remaining sign is
\[
        \chi(I_1,\ldots,I_k)
        =
        \sum_{q<r}|\varphi_{I_r}||\varphi_{I_q}|.
\]

Finally,
$
\langle
\varphi_{I_1},\ldots,\varphi_{I_k}
\rangle_{\varphi_{\emptyset}}
(v_1\cdots v_m)
$
vanishes by definition if $k>m$. Hence the sum over $1\leq k\leq m$ is the
same as the sum over all partitions of $I$. Therefore the coefficient of
$e_I$ in $\varphi(v_1)\cdots\varphi(v_m)$ is
\[
\sum_{\{I_1,\ldots,I_k\}\in \Pi(I)}
(-1)^{\chi(I_1,\ldots,I_k)}
\langle
\varphi_{I_1},\ldots,\varphi_{I_k}
\rangle_{\varphi_{\emptyset}}
(v_1\cdots v_m).
\]
Since this holds for all homogeneous elements $v_1,\ldots,v_m\in V$, the
desired identity follows.
\end{proof}

For a non-empty subset $I\subseteq [p]$, put
\[
\kappa (I)
=
\sum_{\substack{a<b\\ a,b\in I}} n_a n_b .
\]
The sign $(-1)^{\kappa(I)}$ comes from comparing the ordered product
$e_I=e_{i_1}\cdots e_{i_k}$ with the product orientation of
$S^{n_{i_1}}\times\cdots\times S^{n_{i_k}}$.
If $\{I_1,\ldots,I_k\}$ is a partition of $I$, then
\[
\kappa(I)
\equiv
\sum_{r=1}^k \kappa(I_r)
+
\chi(I_1,\ldots,I_k)
\pmod 2.
\]
Indeed, the pairs $(a,b)$ in $I$ are divided into pairs contained in one block
and pairs contained in two different blocks.

\begin{prop}\label{prop:differential-identity}
For $\emptyset\ne I\subsetneq [p]$, let
$
\theta_I\in \Der^{-|e_I|}(\wedge V,A;\varphi_{\emptyset})
$
be the $\varphi_{\emptyset}$-derivation determined by
$
        \theta_I(v)=(-1)^{\kappa(I)}\varphi_I(v)
$
for $v\in V$.
Then
\[
\delta\theta_I
=
-
\sum_{\substack{\{I_1,\ldots,I_k\}\in\Pi(I)\\ k\geq 2}}
[
\theta_{I_1},
\ldots,
\theta_{I_k}
].
\]
Here the blocks $I_1,\ldots,I_k$ are ordered as in
Proposition~\ref{prop:component-decomposition}, and the brackets are taken
with respect to $\varphi_{\emptyset}$.
\end{prop}

\begin{proof}
Let $v\in V$. By definition, we have
\[
        d_A\theta_I(v)=(-1)^{\kappa(I)}d_A\varphi_I(v).
\]
On the other hand, Proposition~\ref{prop:component-decomposition} gives
\begin{align*}
\varphi_I(dv)
&=
\langle \varphi_I\rangle_{\varphi_{\emptyset}}(dv)
+
\sum_{\substack{\{I_1,\ldots,I_k\}\in\Pi(I)\\ k\geq 2}}
(-1)^{\chi(I_1,\ldots,I_k)}
\langle
\varphi_{I_1},\ldots,\varphi_{I_k}
\rangle_{\varphi_{\emptyset}}(dv)       \\
&=
(-1)^{\kappa(I)}\theta_I(dv)
\\
&
\hspace{4em}
-
\sum_{\substack{\{I_1,\ldots,I_k\}\in\Pi(I)\\ k\geq 2}}
(-1)^{\chi(I_1,\ldots,I_k)
+\sum_j\kappa(I_j)+\sum_j|\theta_{I_j}|}
[
\theta_{I_1},\ldots,\theta_{I_k}
](v)                                      \\
&=
(-1)^{\kappa(I)}\theta_I(dv)
-
\sum_{\substack{\{I_1,\ldots,I_k\}\in\Pi(I)\\ k\geq 2}}
(-1)^{\kappa(I)+\sum_j|\theta_{I_j}|}
[
\theta_{I_1},\ldots,\theta_{I_k}
](v).
\end{align*}
In the last equality, we used
\[
\kappa(I)
\equiv
\sum_j\kappa(I_j)+\chi(I_1,\ldots,I_k)
\pmod 2.
\]

Since $\varphi$ is a CDGA morphism, the coefficient of $e_I$ in
$d\varphi(v)=\varphi(dv)$ gives
\[
        d_A\varphi_I(v)=(-1)^{|\varphi_I|}\varphi_I(dv).
\]
Therefore
\begin{align*}
\delta\theta_I(v)
&=
d_A\theta_I(v)-(-1)^{|\theta_I|}\theta_I(dv)                         \\
&=
(-1)^{\kappa(I)+|\varphi_I|}\varphi_I(dv)
-
(-1)^{|\theta_I|}\theta_I(dv)                                         \\
&=
(-1)^{\kappa(I)+|\theta_I|}
\left\{
\varphi_I(dv)-(-1)^{\kappa(I)}\theta_I(dv)
\right\}                                                             \\
&=
-
\sum_{\substack{\{I_1,\ldots,I_k\}\in\Pi(I)\\ k\geq 2}}
[
\theta_{I_1},\ldots,\theta_{I_k}
](v).
\end{align*}
Here we used
$
|\theta_I|\equiv \sum_j|\theta_{I_j}|\pmod 2.
$
Since both sides are $\varphi_{\emptyset}$-derivations, the equality holds on
$\wedge V$.
\end{proof}

\begin{cor}\label{cor:single-component}
For $i=1,\ldots,p$, $\varphi_{\{i\}}$ is a
$\varphi_{\emptyset}$-derivation and satisfies
$\delta\varphi_{\{i\}}=0$.
\end{cor}

\begin{proof}
The derivation property follows from the case $I=\{i\}$ of
Proposition~\ref{prop:component-decomposition}. Since $\kappa(\{i\})=0$, the
derivation $\theta_{\{i\}}$ in Proposition~\ref{prop:differential-identity} is
$\varphi_{\{i\}}$. For $I=\{i\}$, there is no partition into at least two
non-empty blocks. Hence Proposition~\ref{prop:differential-identity} gives
$\delta\varphi_{\{i\}}=0$.
\end{proof}

\subsection{Strict representatives}

Let $\mathcal H_W$ be the rational cohomology algebra of the wedge sum $W$ and
$
\rho:\mathcal H_T\to \mathcal H_W
$
the CDGA morphism introduced in Section 3.
The next lemma allows us to choose a convenient representative
$
\varphi:\wedge V\to \mathcal H_T\otimes A.
$
It says that, if the restriction of a representative over $\mathcal H_T$ is
homotopic to a given map $\psi$, then we may replace the representative by a
homotopic one whose restriction is exactly $\psi$.

\begin{lem}\label{lem:strict-cohomology-model}
Let $\varphi':\wedge V\to \mathcal H_T\otimes A$ be a CDGA morphism.
Suppose that $(\rho\otimes 1)\varphi'$ is homotopic to a CDGA morphism
$
\psi:\wedge V\to \mathcal H_W\otimes A.
$
Then there exists a CDGA morphism
$
\varphi:\wedge V\to \mathcal H_T\otimes A
$
which is homotopic to $\varphi'$ and satisfies
$
(\rho\otimes 1)\varphi=\psi.
$
\end{lem}

\begin{proof}
Choose a CDGA homotopy
\[
H':
\wedge V
\longrightarrow
\mathcal H_W\otimes A\otimes \Lambda(t,dt)
\]
from $(\rho\otimes 1)\varphi'$ to $\psi$. Thus
$
\operatorname{ev}_0H'=(\rho\otimes 1)\varphi'
$
and
$
\ev_1H'=\psi.
$
Here $\ev_i:\Lambda(t,dt)\to \mathbb Q$ denotes the CDGA map given by $\operatorname{ev}_i (t) = i$.

It is enough to construct a CDGA homotopy
\[
H:
\wedge V
\longrightarrow
\mathcal H_T\otimes A\otimes \Lambda(t,dt)
\]
such that
$
\operatorname{ev}_0H=\varphi'
$
and
$
(\rho\otimes 1\otimes 1)H=H'.
$
The homotopy $H$ will be defined on the generators as follows.
For $v\in V$, write
\[
H'(v)
=
1\otimes H'_\emptyset(v)
+
\sum_{i=1}^p e_i\otimes H'_i(v).
\]
Define the components of $H(v)$ by
\[
H(v)
=
1\otimes H_\emptyset(v)
+
\sum_{i=1}^p e_i\otimes H_i(v)
+
\sum_{\substack{I\subsetneq[p]\\ |I|\geq 2}}
e_I\otimes H_I(v),
\]
where
$
H_\emptyset(v)=H'_\emptyset(v)
$
and
$
H_i(v)=H'_i(v)
$
for $i=1,\ldots,p$. It remains to define $H_I(v)$ for $|I|\geq 2$.
The construction proceeds by induction over an ordered basis of $V$ compatible
with the Sullivan filtration. Suppose first that $dv=0$.
Write
\[
        \varphi'(v)=\sum_{I\subsetneq[p]} e_I\otimes \varphi'_I(v).
\]
Then each $\varphi'_I(v)$ is a cocycle. In this case, for $|I|\geq 2$, define
\[
        H_I(v)=\varphi'_I(v)\otimes 1.
\]
Now suppose that $dv\neq 0$. Let $V_{<v}$ be the subspace spanned by the
generators preceding $v$.
By induction, assume that $H$ has been constructed
as a CDGA morphism on $\wedge V_{<v}$ and satisfies
$
\ev_0H=\varphi'
$
and
$
(\rho\otimes 1\otimes 1)H=H'
$
on $\wedge V_{<v}$. Since $dv\in\wedge V_{<v}$, the element $H(dv)$ is
already defined.
For $I\subsetneq[p]$ with $|I|\geq 2$, let $H_I(dv)$ be the
coefficient of $e_I$ in $H(dv)$, and put
\[
        \alpha_I=(-1)^{|e_I|}H_I(dv).
\]
Since $H$ is a CDGA morphism on $\wedge V_{<v}$,
$
        DH(dv)=H(d^2v)=0,
$
where $D$ denotes the differential of $A\otimes\Lambda(t,dt)$. 
Thus $\alpha_I$ is a cocycle in $A\otimes\Lambda(t,dt)$. 
Moreover, the induction hypothesis gives
\[
\operatorname{ev}_0(\alpha_I)
=
(-1)^{|e_I|}\varphi'_I(dv)
=
d_A\varphi'_I(v),
\]
where the last equality follows because $\varphi'$ is a CDGA morphism.

We use the standard
homotopy operator
$
K:A\otimes\Lambda(t,dt)\to A\otimes\Lambda(t,dt)
$
of degree $-1$ associated with evaluation at $t=0$. Thus
\[
        DK+KD=\operatorname{id}-\iota\operatorname{ev}_0
        \quad
        \text{and}
        \quad
        \operatorname{ev}_0K=0,
\]
where $\iota:A\to A\otimes\Lambda(t,dt)$ is given by $a\mapsto a\otimes 1$.
Here $K$ is used only as a linear map on the underlying cochain complex.
Put
\[
        \beta_I=\alpha_I-D(\varphi'_I(v)\otimes 1).
\]
Then $\beta_I$ is a cocycle and $\ev_0(\beta_I)=0$. Hence
$DK(\beta_I)=\beta_I$. Define
\[
H_I(v)
=
\varphi'_I(v)\otimes 1
+
K(\beta_I).
\]
Since $\ev_0K=0$, we have
$
\operatorname{ev}_0H_I(v)=\varphi'_I(v)
$
and
\[
        DH_I(v)
        =
        D(\varphi'_I(v)\otimes 1)+\beta_I
        =
        \alpha_I
        =
        (-1)^{|e_I|}H_I(dv).
\]
Since $\wedge V$ is free, this choice of $H(v)$ extends the partially defined
map uniquely to a graded algebra map.
The identity
$DH(v)=H(dv)$ makes this extension a CDGA morphism.
Continuing the
induction gives the required CDGA homotopy $H$.
Finally, set
\[
        \varphi=\ev_1 H.
\]
Then $\varphi$ is homotopic to $\varphi'$, and
\[
(\rho\otimes 1)\varphi
=
\ev_1 \bigl((\rho\otimes 1\otimes 1)H\bigr)
=
\ev_1 H'
=
\psi.
\]
This completes the proof.
\end{proof}


\section{Defining systems and higher order Whitehead products in mapping spaces}

In this section, we apply the algebraic results of Section 4 to mapping spaces.
Let $X$ be a simply connected finite CW complex, $Y$ a simply connected
space of finite type, and $f\colon X\to Y$ a map. Write
$
        \ell : Y\longrightarrow Y_{\mathbb Q}
$
for the rationalization. Since $X$ is finite, composition with $\ell$ induces a
rationalization
\[
\ell_*:
\map(X,Y;f)
\longrightarrow
\map(X,Y_{\mathbb Q};\ell\circ f);
\]
see \cite[Chapter~II, Theorem~3.11 and the following paragraph]{HMR1975}.
Thus
\[
\pi_n(\map(X,Y;f))\otimes\mathbb Q
\cong
\pi_n(\map(X,Y_{\mathbb Q};\ell\circ f)).
\]
We define rational higher order Whitehead products in $\map(X,Y;f)$ as the
Porter products in the component on the right.

For the rest of this section, we write $Y$ for $Y_{\mathbb Q}$ and $f$ for
$\ell\circ f$. Let $(\wedge V,d)$ be a minimal Sullivan model for $Y$,
$(A,d_A)$ a CDGA model for $X$, and $\phi : (\wedge V,d) \to (A,d_A)$ a Sullivan representative for $f$.
By the Sullivan model description of the rational homotopy groups of mapping
spaces due to Lupton--Smith~\cite[Theorem~2.1]{LS2007}, we use the standard
identification
\begin{equation}\label{eq:adjoint-derivation-isomorphism}
        \Phi:\pi_n(\map(X,Y;f))\otimes\mathbb{Q}
        \cong
        H^{-n}(\Der(\wedge V,A;\phi)).
\end{equation}
We use cohomological grading, so that derivations corresponding to $\pi_n$
have degree $-n$.
To recall how this isomorphism is defined, let
$
\alpha\in \pi_n(\map(X,Y;f))\otimes\mathbb{Q}
$
be represented by a map
$
g:S^n\to\map(X,Y;f).
$
Let
\[
        \ad_g:S^n\times X\longrightarrow Y
\]
be the adjoint map. If
$
\varphi_g:\wedge V\to H^*(S^n;\mathbb{Q})\otimes A
$
is a Sullivan representative for $\ad_g$, then we may write
\[
        \varphi_g(v)
        =
        1\otimes \phi(v)+e_n\otimes \theta_g(v)
        \qquad (v\in V),
\]
where $e_n$ is the fundamental cohomology class of $S^n$. Then $\theta_g$ is a
$\phi$-derivation cycle of degree $-n$, and $\Phi(\alpha)$ is represented by
$\theta_g$. By abuse of notation, we use the same symbol for a cycle and for
its cohomology class when no confusion can occur.

\begin{defn}\label{def:defining-system}
Let $\theta_i\in\Der^{-n_i}(\wedge V,A;\phi)$ be cycles for
$i=1,\ldots,p$. For $I\subseteq[p]$, set
$n_I=\sum_{i\in I}n_i$.
A \emph{defining system} for $\theta_1,\ldots,\theta_p$ is a family
\[
        \Theta=\{\theta_I\}_{\emptyset\neq I\subsetneq[p]}
\]
of $\phi$-derivations
$
\theta_I\in \Der^{-n_I}(\wedge V,A;\phi)
$
such that $\theta_{\{i\}}=\theta_i$ and, for every subset
$I\subsetneq[p]$ with $|I|\geq 2$,
\[
\delta\theta_I
=
-
\sum_{\substack{\{I_1,\ldots,I_k\}\in\Pi(I)\\ k\geq 2}}
[
\theta_{I_1},\ldots,\theta_{I_k}
].
\]
Here $\Pi(I)$ denotes the set of all partitions of $I$ into non-empty subsets.
For each partition, the blocks are ordered by
$
\min I_1<\cdots<\min I_k.
$
The brackets are taken with respect to $\phi$.
\end{defn}

For a defining system
$\Theta=\{\theta_I\}_{\emptyset\neq I\subsetneq[p]}$, define its associated bracket by
\[
[\theta_1,\ldots,\theta_p]_\Theta
=
\sum_{\substack{
\{I_1,\ldots,I_k\}\in\Pi([p])\\
k\geq2}}
[\theta_{I_1},\ldots,\theta_{I_k}].
\]
Here the same ordering convention for partitions is used.
The associated bracket has degree $1-N$, where
$N=n_1+\cdots+n_p$.
It will be shown in the proof of Theorem~\ref{thm:main} that this bracket is
a cycle. We then use the same symbol for its cohomology class.

Under the rationalization convention above,
$[\alpha_1,\ldots,\alpha_p]_{\mathrm{Wh}}$ is regarded as a subset of
$\pi_{N-1}(\map(X,Y;f))\otimes\mathbb Q$.
The following theorem is the main result of this paper.

\begin{thm}\label{thm:main}
Let $p\geq2$. Suppose that integers $n_i\geq2$, classes
$
        \alpha_i\in
        \pi_{n_i}(\map(X,Y;f))\otimes\mathbb Q,
$
and cycles
$
        \theta_i\in\Der^{-n_i}(\wedge V,A;\phi)
$
representing $\Phi(\alpha_i)$ are given for $i=1,\ldots,p$.
If the higher order Whitehead product
$
        [\alpha_1,\ldots,\alpha_p]_{\mathrm{Wh}}
$
is non-empty, then
\[
\Phi\bigl([\alpha_1,\ldots,\alpha_p]_{\mathrm{Wh}}\bigr)
=
\left\{
[\theta_1,\ldots,\theta_p]_\Theta
\ \middle|\
\Theta \text{ is a defining system for } \theta_1,\ldots,\theta_p
\right\}.
\]
\end{thm}

\begin{proof}
Let
$
\beta\in[\alpha_1,\ldots,\alpha_p]_{\mathrm{Wh}}.
$
Then there is an extension
\[
\widetilde g:
T=T(S^{n_1},\ldots,S^{n_p})
\longrightarrow
\map(X,Y;f)
\]
of the map
$
g:W\to \map(X,Y;f)
$
determined by $\alpha_1,\ldots,\alpha_p$ such that $\beta$ is represented by $\widetilde g\circ w$, where $w:S^{N-1}\to T$ is the universal higher order Whitehead product.

Let
$
\ad_g:W\times X\to Y
$
and
$
\ad_{\widetilde g}:T\times X\to Y
$
be the adjoint maps of $g$ and $\widetilde g$, respectively. Then the diagram
\[
\xymatrix{
T \times X \ar[r]^-{\ad_{\widetilde g}} & Y \\
W \times X \ar@{^{(}->}[u] \ar[ru]_-{\ad_g} &
}
\]
commutes up to homotopy.
Define
\[
\psi
=
1\otimes\phi
+
\sum_{i=1}^p e_i\otimes\theta_i
:
\wedge V \longrightarrow  \mathcal H_W\otimes A.
\]
Since each $\theta_i$ is a $\phi$-derivation cycle and all products of
positive-degree elements in $\mathcal H_W$ vanish, the map $\psi$ is a CDGA
morphism. By the construction of the identification $\Phi$ in
\eqref{eq:adjoint-derivation-isomorphism}, $\psi$ is homotopic to a Sullivan
representative for $\ad_g$; compare \cite[Section~2]{Na2014}.

Let
$
\varphi':\wedge V\to \mathcal H_T\otimes A
$
be a Sullivan representative for $\ad_{\widetilde g}$. Since the restriction of
$\ad_{\widetilde g}$ to $W\times X$ is homotopic to $\ad_g$, the map
$
(\rho\otimes 1)\varphi'
$
is homotopic to $\psi$. By Lemma~\ref{lem:strict-cohomology-model}, we may
replace $\varphi'$ by a homotopic Sullivan representative
\[
\varphi:\wedge V\longrightarrow \mathcal H_T\otimes A
\]
such that
$
(\rho\otimes 1)\varphi=\psi.
$

Recall the quasi-isomorphism
$q_T:\mathcal M_T\to\mathcal H_T$
introduced in Section~3. It sends each $e_i$ to its cohomology class and all
other generators to zero.
By the lifting lemma
\cite[Lemma~12.4]{FHT}, there exists a lift
$
\widetilde\varphi:\wedge V\to \mathcal M_T\otimes A
$
such that the following diagram is strictly commutative:
\[
\xymatrix{
& \mathcal M_T\otimes A \ar[d]^-{q_T\otimes 1} \\
\wedge V \ar[r]^-{\varphi} \ar[ru]^-{\widetilde\varphi}
& \mathcal H_T\otimes A.
}
\]

The lift may be chosen so that no pure $e_{[p]}$-term occurs in
$\widetilde\varphi(v)$ for $v\in V$.
Indeed, choose an ordered basis of $V$ compatible with the Sullivan
filtration, and construct the lift inductively.
Suppose that the lift has been defined on $\wedge V_{<v}$.
The usual lifting argument first gives a preliminary value for
$\widetilde\varphi(v)$ satisfying
\[
d\widetilde\varphi(v)=\widetilde\varphi(dv),
\qquad
(q_T\otimes1)\widetilde\varphi(v)=\varphi(v).
\]
If its pure $e_{[p]}$-component is $e_{[p]}\otimes\mu_v$, replace
$\widetilde\varphi(v)$ by
\[
\widetilde\varphi(v)-d(u_{[p]}\otimes\mu_v).
\]
Since
$
d(u_{[p]}\otimes\mu_v)
=
e_{[p]}\otimes\mu_v
+
(-1)^{N-1}u_{[p]}\otimes d_A\mu_v$,
this replacement removes the pure $e_{[p]}$-component. It does not change
the differential of $\widetilde\varphi(v)$, and it does not change its image
under $q_T\otimes1$, since $q_T(u_{[p]})=0$.
Continuing the induction gives a lift with the required property.
Therefore we may write
\[
\widetilde\varphi(v)
=
1\otimes\phi(v)
+
\sum_{\emptyset\neq I\subsetneq[p]}
e_I\otimes\varphi_I(v)
+
u_{[p]}\otimes\xi(v)
+
R(v),
\]
where $R(v)$ is the sum of all remaining monomials in
$\mathcal M_T\otimes A$. Only a pure $u_{[p]}$-term can contribute to the
pure $e_{[p]}$-component, and that term has already been separated.
For each non-empty proper subset $I\subsetneq[p]$, let $\theta_I$ be the
$\phi$-derivation determined by
\[
        \theta_I(v)=(-1)^{\kappa(I)}\varphi_I(v)
        \qquad (v\in V).
\]
Then
$
\Theta=\{\theta_I\}_{\emptyset\neq I\subsetneq[p]}
$
is a defining system for $\theta_1,\ldots,\theta_p$ by
Proposition~\ref{prop:differential-identity}. 
Since $\kappa(\{i\})=0$, the equality
$
(\rho\otimes 1)\varphi=\psi
$
gives $\theta_{\{i\}}=\theta_i$ for $i=1,\ldots,p$.

Since $\widetilde\varphi$ is a CDGA morphism, comparing the coefficient of
$e_{[p]}$ gives
\begin{align*}
\xi(v)
&=
\sum_{\substack{\{I_1,\ldots,I_k\}\in\Pi([p])\\ k\geq 2}}
(-1)^{\chi(I_1,\ldots,I_k)}
\langle
\varphi_{I_1},\ldots,\varphi_{I_k}
\rangle_{\phi}(dv)                                      \\
&=
(-1)^{N+1+\kappa([p])}
[\theta_1,\ldots,\theta_p]_\Theta(v).
\end{align*}

Let
$
\omega:\mathcal M_T\to\mathcal M_{S^{N-1}}
$
be a Sullivan representative for $w$.
Proposition \ref{prop:model_universal-HOP}
gives
\[
        \omega(u_{[p]})
        =
        (-1)^{N-1+\kappa([p])}e
        \quad
        \text{and}
        \quad
        Q(\omega)|_{V_P^{N-1}}=0.
\]
Let
$
q_S:\mathcal M_{S^{N-1}}
\to H^*(S^{N-1};\mathbb Q)
$
be the standard quasi-isomorphism, and write $q_S(e)=e_{N-1}$.
Since $q_S$ also kills all decomposable and higher-degree terms, only the
$u_{[p]}$-term contributes.
Hence
\[
\begin{aligned}
\bigl((q_S\circ\omega)\otimes1\bigr)\widetilde\varphi(v)
&=
1\otimes\phi(v)
+
(-1)^{N-1+\kappa([p])}
e_{N-1}\otimes\xi(v)
\\
&=
1\otimes\phi(v)
+
e_{N-1}\otimes
[\theta_1,\ldots,\theta_p]_\Theta(v).
\end{aligned}
\]
By the definition of the identification
\eqref{eq:adjoint-derivation-isomorphism}, we obtain
\[
        \Phi(\beta)
=
[\theta_1,\ldots,\theta_p]_\Theta.
\]

Conversely, let
$
\Theta=\{\theta_I\}_{\emptyset\neq I\subsetneq[p]}
$
be a defining system for $\theta_1,\ldots,\theta_p$. 
Define a graded algebra
map
$
\varphi_\Theta:\wedge V\to\mathcal H_T\otimes A
$
by
\[
\varphi_\Theta(v)
=
1\otimes\phi(v)
+
\sum_{\emptyset\neq I\subsetneq[p]}
(-1)^{\kappa(I)}e_I\otimes\theta_I(v)
\qquad (v\in V).
\]
The defining system identities imply, by the same coefficient calculation as
in Proposition~\ref{prop:differential-identity}, that
$
        d\varphi_\Theta(v)=\varphi_\Theta(dv)
$
for $v\in V$. Hence $\varphi_\Theta$ is a CDGA morphism.

By the Sullivan-de Rham correspondence, $\varphi_\Theta$ represents a rational
homotopy class of maps
$
g'_\Theta:T\times X\to Y.
$
Let
\[
\widetilde g_\Theta:T\longrightarrow \map(X,Y;f)
\]
be the adjoint map of $g'_\Theta$. Since $\kappa(\{i\})=0$, we have
\[
(\rho\otimes 1)\varphi_\Theta
=
1\otimes\phi
+
\sum_{i=1}^p e_i\otimes\theta_i
=
\psi.
\]
Thus the restriction of $\widetilde g_\Theta$ to $W$ is homotopic to $g$.
Since the inclusion $W\hookrightarrow T$ is a cofibration, the homotopy
extension property allows us to replace $\widetilde g_\Theta$ by a homotopic
map whose restriction to $W$ is exactly $g$. We keep the same notation for
this map.
Hence the composite
\[
        \beta_\Theta=\widetilde g_\Theta\circ w
\]
represents an element of
$[\alpha_1,\ldots,\alpha_p]_{\mathrm{Wh}}$.

Choose a lift
$
\widetilde\varphi_\Theta:\wedge V\to\mathcal M_T\otimes A
$
of $\varphi_\Theta$ with respect to $q_T\otimes 1$, again with no pure
$e_{[p]}$-term in $\widetilde\varphi_\Theta(v)$ for $v\in V$.
Repeating the coefficient comparison and applying $q_S\otimes1$ gives
\[
\bigl((q_S\circ\omega)\otimes1\bigr)
\widetilde\varphi_\Theta(v)
=
1\otimes\phi(v)
+
e_{N-1}\otimes
[\theta_1,\ldots,\theta_p]_\Theta(v).
\]
Since
$
\bigl((q_S\circ\omega)\otimes1\bigr)
\widetilde\varphi_\Theta$ is a CDGA morphism, the coefficient of $e_{N-1}$ is a
$\phi$-derivation cycle. Hence
$
\delta[\theta_1,\ldots,\theta_p]_\Theta=0.
$
By the definition of the identification \eqref{eq:adjoint-derivation-isomorphism},
we obtain
\[
\Phi(\beta_\Theta)
=
[\theta_1,\ldots,\theta_p]_\Theta.
\]
This proves the theorem.
\end{proof}

When $X=*$, the mapping space is naturally identified with $Y$, and we may take $A=\mathbb Q$ and $\phi$ to be the augmentation of $\wedge V$.
In this case, $\Der(\wedge V,\mathbb Q;\phi)$ is naturally identified with $\Hom(V,\mathbb Q)$, and the identification $\Phi$ agrees with the standard identification
$
\pi_n(Y)\otimes\mathbb Q
\cong
\Hom(V^n,\mathbb Q).
$
Let $\mu\in V^{N-1}$.
If every monomial in $d\mu$ has word length at least $p$, then every term in the associated bracket arising from a partition with fewer than $p$ blocks vanishes on $\mu$.
Hence, when evaluated on $\mu$, only the direct bracket
$[\theta_1,\ldots,\theta_p]$ contributes.
In this way, Theorem~\ref{thm:main} recovers the formula of Andrews and Arkowitz \cite[Theorem~5.4]{AA1978} up to the sign convention explained in Remark~\ref{rem:sign}.
Without this assumption, the additional terms involving the higher components of a defining system describe the possible indeterminacy of the higher order Whitehead product.

Another special case of Theorem~A is obtained by taking $X=Y$ and
$f=\id_Y$, where $Y$ is a simply connected finite CW complex.
Then $\map(Y,Y;\id_Y)=\aut_1(Y)$, which is a topological monoid under
composition.
Let $g:W\to\aut_1(Y)$ be the map determined by
$\alpha_1,\ldots,\alpha_p$, and write
$g_i:S^{n_i}\to\aut_1(Y)$ for its restriction to the $i$-th sphere.
Define $\widehat g:P\to\aut_1(Y)$ by
\[
        \widehat g(s_1,\ldots,s_p)
        =
        g_1(s_1)\circ\cdots\circ g_p(s_p).
\]
Since each $g_i$ is based at $\id_Y$, the restriction of $\widehat g$
to $W$ is $g$. Hence $\widetilde g=\widehat g|_T$ is an extension of
$g$ over the fat wedge. Moreover,
\[
        \widetilde g\circ w
        =
        \widehat g\circ j_T\circ w
        \simeq *,
\]
since $j_T\circ w$ is null-homotopic. Therefore
\[
        0\in[\alpha_1,\ldots,\alpha_p]_{\mathrm{Wh}}.
\]

This vanishing also has a direct interpretation in the derivation
complex. Let $(\wedge V,d)$ be a minimal Sullivan model for $Y$, and
take $A=(\wedge V,d)$ and $\phi=\id_{\wedge V}$. By the same
strictification argument as in
Lemma~\ref{lem:strict-cohomology-model}, we may choose a Sullivan
representative
\[
        \varphi^P:
        (\wedge V,d)
        \longrightarrow
        \mathcal{H}_P\otimes(\wedge V,d)
\]
for the adjoint of $\widehat g$. With respect to the basis
$\{e_I\}_{I\subseteq[p]}$ of $\mathcal{H}_P$, write
\[
        \varphi^P
        =
        \sum_{I\subseteq[p]}e_I\otimes\varphi_I,
\]
where $\varphi_{\emptyset}=\phi$ and
$\varphi_{\{i\}}|_V=\theta_i|_V$ for $i=1,\ldots,p$. For each non-empty
subset $I\subseteq[p]$, let $\theta_I$ be the $\phi$-derivation
determined by
$
        \theta_I(v)=(-1)^{\kappa(I)}\varphi_I(v)
$
for $v\in V$. The same coefficient calculation as in
Proposition~\ref{prop:differential-identity} shows that the components
indexed by the non-empty proper subsets of $[p]$ form a defining
system
$
        \Theta=\{\theta_I\}_{\emptyset\neq I\subsetneq[p]},
$
while the top component satisfies
\[
        \delta\theta_{[p]}
        =
        -[\theta_1,\ldots,\theta_p]_{\Theta}.
\]
Thus the associated bracket is exact and represents zero in
cohomology. By Theorem~\ref{thm:main}, this algebraically recovers
\[
        0\in[\alpha_1,\ldots,\alpha_p]_{\mathrm{Wh}}.
\]

Buijs and Murillo considered the Sullivan model $(\wedge S_\phi,\bar d)$ of a mapping-space component obtained from the Brown--Szczarba model.
They showed that the $j$-th homogeneous part $\bar d_j$ of its differential
corresponds to the bracket $[\varphi_1,\ldots,\varphi_j]$ of $j$ derivations \cite[Theorem~15]{BM2008}.
These are the brackets recalled in Subsection~4.1.
Definition~\ref{def:defining-system} uses them in the differential identities satisfied by the components $\theta_I$ of a defining system.
The associated bracket of a defining system is then obtained by summing the brackets corresponding to all partitions of $[p]$ into at least two blocks.
By allowing the defining system to vary,
Theorem~\ref{thm:main} describes the entire higher order Whitehead product whenever it is non-empty, including its indeterminacy.
Example~\ref{ex:quadratic-cp2} illustrates this point: the direct third-order bracket vanishes, while a suitable defining system gives a non-zero associated bracket.

\section{Examples}

In this section, we give examples that illustrate how defining systems affect
higher order Whitehead products in mapping spaces.

The first example is a basic one. It shows that a mapping space can have a
non-trivial third-order Whitehead product even though the defining relation in
the target has order five. In the minimal model of $\mathbb{C}P^4$, this
relation is given by $x^5$. For the inclusion
$\mathbb{C}P^2\to \mathbb{C}P^4$, however, two powers of $x$ are already
supplied by the map $f$. As a result, a third-order bracket appears in the
mapping space.

\begin{ex}\label{ex:cp2-cp4-inclusion}
Let $X=\mathbb{C}P^2$ and $Y=\mathbb{C}P^4$.
Put
\[
        A=H^*(\mathbb{C}P^2;\mathbb{Q})
        =\mathbb{Q}[a]/(a^3),
        \qquad |a|=2.
\]
Regard $A$ as a CDGA with zero differential.
A minimal Sullivan model for $Y$ is
$(\wedge V,d)=(\wedge(x,y),d)$, where
$|x|=2$, $|y|=9$, $dx=0$, and $dy=x^5$.
Let $f:\mathbb{C}P^2\hookrightarrow \mathbb{C}P^4$ be the standard inclusion.
A Sullivan representative for $f$ is the CDGA morphism
$\phi:\wedge V\to A$ given by $\phi(x)=a$ and $\phi(y)=0$.

Define a derivation $\theta\in \Der^{-2}(\wedge V,A;\phi)$ by
$\theta(x)=1$ and $\theta(y)=0$. Then $\theta$ is a cycle.
The second-order bracket is
\[
[\theta,\theta](y)
=
-\langle \theta,\theta\rangle_\phi(x^5)
=
-20 a^3
=
0.
\]
Thus, for the three inputs
$\theta_1=\theta_2=\theta_3=\theta$, all second-order brackets vanish.
Hence the family $\Theta$ with $\theta_{\{i\}}=\theta$ for $i=1,2,3$ and
$\theta_{12}=\theta_{13}=\theta_{23}=0$ is a defining system.
For this defining system,
\[
[\theta_1,\theta_2,\theta_3]_\Theta
=
[\theta,\theta,\theta].
\]

Evaluating this on $y$ yields
$
[\theta,\theta,\theta](y)=-60a^2.
$
Thus $[\theta,\theta,\theta]$ is non-zero as a derivation.
It also represents a non-trivial cohomology class. Indeed, if
$\xi\in \Der^{-6}(\wedge V,A;\phi)$, then
$
(\delta\xi)(y)=-\xi(x^5)=0,
$
since $\xi(x)\in A^{-4}=0$.
Therefore no boundary in degree $-5$ can take a non-zero value on $y$.
Let $\zeta \in \Der^{-5}(\wedge V,A;\phi)$ be defined by
$\zeta(x)=0$ and $\zeta(y)=a^2$. Then $\zeta$ is a cycle, and the algebraic computation
above gives
\[
        [\theta,\theta,\theta]=-60\zeta.
\]

Let
\[
\alpha \in \pi_2(\map(\mathbb{C}P^2,\mathbb{C}P^4;f))\otimes\mathbb{Q},
\qquad
\beta \in \pi_5(\map(\mathbb{C}P^2,\mathbb{C}P^4;f))\otimes\mathbb{Q}
\]
correspond to the cohomology classes represented by $\theta$ and $\zeta$, respectively.
Since $[\theta,\theta]=0$, the corresponding rational second-order Whitehead
product vanishes. Hence the third-order Whitehead product
$[\alpha,\alpha,\alpha]_{\mathrm{Wh}}$ is non-empty.

Moreover, in this example, we see that $\Der^{-4}(\wedge V,A;\phi)=0$ for degree reasons.
Hence the above trivial defining system is the only defining system for the three
inputs $\theta,\theta,\theta$.
Therefore, by Theorem~\ref{thm:main}, the rational third-order Whitehead product
$[\alpha,\alpha,\alpha]_{\mathrm{Wh}}$ consists of the single rational homotopy
class $-60\beta$.
In particular, this product is non-zero.

We now give a geometric representative for $\alpha$.
Define
\[
g_\alpha:\mathbb CP^1
\longrightarrow
\map(\mathbb CP^2,\mathbb CP^4;f)
\]
by taking its adjoint $F_\alpha:\mathbb{C}P^1\times \mathbb{C}P^2 \to \mathbb{C}P^4$
to be
\[
F_\alpha([s:t],[\lambda_0:\lambda_1:\lambda_2])
=
[s\lambda_0 : s\lambda_1: s\lambda_2+t\lambda_0 : t\lambda_1 : t\lambda_2].
\]
At the base point $[1:0]\in \mathbb{C}P^1$, this restricts to the standard
inclusion $f$. Thus $g_\alpha$ is based at $f$.
Let $h$ and $a$ be the standard generators of
$H^2(\mathbb{C}P^1;\mathbb{Q})$ and
$H^2(\mathbb{C}P^2;\mathbb{Q})$, respectively.
Since
$H^2(\mathbb{C}P^1\times \mathbb{C}P^2;\mathbb{Q})= \mathbb{Q}(h\otimes 1)\oplus \mathbb{Q}(1\otimes a)$,
we can write
\[
F_\alpha^*(x)=c_1(h\otimes 1)+c_2(1\otimes a) \in H^2(\mathbb{C}P^1\times \mathbb{C}P^2;\mathbb{Q})
\]
for some $c_1,c_2\in\mathbb{Q}$.
If $[\lambda_0:\lambda_1:\lambda_2]=[1:0:0]$, then
$F_\alpha([s:t],[1:0:0])=[s:0:t:0:0]$,
which is a coordinate linear inclusion
$\mathbb{C}P^1\hookrightarrow \mathbb{C}P^4$; hence $c_1=1$.
If $[s:t]=[1:0]$, then
$
F_\alpha([1:0],[\lambda_0:\lambda_1:\lambda_2])
=
[\lambda_0:\lambda_1:\lambda_2:0:0]$, which is the standard inclusion
$\mathbb{C}P^2\hookrightarrow \mathbb{C}P^4$; hence $c_2=1$.
Therefore
\[
F_\alpha^*(x)=h\otimes 1+1\otimes a.
\]
Hence a Sullivan representative for $F_\alpha$ is given by
\[
x\longmapsto h\otimes 1+1\otimes a,\qquad y\longmapsto 0.
\]
The coefficient of $h$ is exactly the derivation $\theta$. Hence $g_\alpha$
represents the rational homotopy class $\alpha$.
\end{ex}

The next example shows that defining systems may depend on the component of
the mapping space. For the same spaces $X$ and $Y$, the constant component
admits a trivial defining system. In another component, the second-order
brackets are non-zero but exact, so the higher components of a defining system
must be non-zero. Different choices of these components produce non-trivial
indeterminacy.

\begin{ex}\label{ex:cp4-indeterminacy}
Let $X=\mathbb{C}P^4$. Then
$
A=H^*(\mathbb{C}P^4;\mathbb{Q})=\mathbb{Q}[a]/(a^5)
$
with $|a|=2$.
Regard $A$ as a CDGA with zero differential.
Let $Y$ be a rational space with minimal Sullivan model
\[
(\wedge V,d)=(\wedge(x_1,x_2,x_3,z),d),
\]
where
$
|x_1|=|x_2|=|x_3|=4
$
and $|z|=11$, with
$dx_i=0 \ (i=1,2,3)$, $dz=x_1x_2x_3$.

For $\varepsilon\in\{0,1\}$, define a CDGA morphism
\[
\phi_\varepsilon:\wedge V\longrightarrow A
\]
by $\phi_\varepsilon(x_i)=\varepsilon a^2$,  $\phi_\varepsilon(z)=0$.
This is well-defined since $a^6=0$ in $A$.
Let $f_\varepsilon:X\to Y$ be the map represented by
$\phi_\varepsilon$. In particular, $f_0$ is the constant map.

For $i=1,2,3$, define
$
\theta_i^\varepsilon
\in
\Der^{-2}(\wedge V,A;\phi_\varepsilon)
$
by
\[
\theta_i^\varepsilon(x_j)
=
\begin{cases}
a & (i=j),\\
0 & (i\neq j),
\end{cases}
\qquad
\theta_i^\varepsilon(z)=0.
\]
We denote by $\delta_\varepsilon$ the differential of
$\Der(\wedge V,A;\phi_\varepsilon)$.
Then each $\theta_i^\varepsilon$ is a cycle. Indeed,
\[
(\delta_\varepsilon\theta_i^\varepsilon)(z)
=
-\theta_i^\varepsilon(x_1x_2x_3)
=
-\varepsilon^2a^5
=
0.
\]
A direct computation gives, for $1\leq i<j\leq3$,
\[
[\theta_i^\varepsilon,\theta_j^\varepsilon](z)
=
-\varepsilon a^4.
\]

We first consider the component of $f_0$.
In this case,
$
[\theta_i^0,\theta_j^0]=0
$
for
$1\leq i<j\leq3$.
Hence the family $\Theta_0$ with
\[
\theta_{\{i\}}^0=\theta_i^0
\quad(i=1,2,3),
\qquad
\theta_{12}^0=\theta_{13}^0=\theta_{23}^0=0
\]
is a defining system.
More generally, for any defining system $\Theta$ for
$\theta_1^0,\theta_2^0,\theta_3^0$, all mixed terms in the associated bracket
vanish. Indeed, they vanish on the $x_i$ since $dx_i=0$, and their values on
$z$ contain a factor $\phi_0(x_\ell)=0$.
Therefore
\[
[\theta_1^0,\theta_2^0,\theta_3^0]_\Theta
=
[\theta_1^0,\theta_2^0,\theta_3^0],
\]
and
$[\theta_1^0,\theta_2^0,\theta_3^0]_\Theta(z)
=
-a^3.
$

We next consider the component of $f_1$.
To simplify notation, write
\[
\phi=\phi_1,
\qquad
\theta_i=\theta_i^1,
\qquad
\delta=\delta_1.
\]
Then
\[
[\theta_1,\theta_2](z)
=
[\theta_1,\theta_3](z)
=
[\theta_2,\theta_3](z)
=
-a^4
\neq0.
\]
Thus the second-order brackets are non-zero as derivations. However, they
are exact.
For each $1\leq i<j\leq3$,  choose
$
c_{ij,1},c_{ij,2},c_{ij,3}\in\mathbb{Q}
$
such that
\[
        c_{ij,1}+c_{ij,2}+c_{ij,3}=1.
\]
Define
$
\eta_{ij}\in \Der^{-4}(\wedge V,A;\phi)
$
by
$
\eta_{ij}(x_\ell)=c_{ij,\ell}
$
for $\ell=1,2,3$, and by $\eta_{ij}(z)=0$. 
Then
\[
(\delta\eta_{ij})(z)
=
-\eta_{ij}(dz)
=
-\eta_{ij}(x_1x_2x_3)
=
-(c_{ij,1}+c_{ij,2}+c_{ij,3})a^4
=
-a^4.
\]
Hence $\delta\eta_{ij}=[\theta_i,\theta_j]$.

Following our convention for defining systems, set
$
\theta_{ij}=-\eta_{ij}.
$
Then
$
\delta\theta_{ij}=-[\theta_i,\theta_j].
$
Therefore
\[
\Theta_1
=
\{\theta_1,\theta_2,\theta_3,
  \theta_{12},\theta_{13},\theta_{23}\}
\]
is a defining system for $\theta_1,\theta_2,\theta_3$.
Its associated third-order bracket is
\[
[\theta_1,\theta_2,\theta_3]_{\Theta_1}
=
[\theta_1,\theta_2,\theta_3]
+
[\theta_{12},\theta_3]
+
[\theta_{13},\theta_2]
+
[\theta_1,\theta_{23}].
\]
Evaluating each term on $z$, we obtain
\begin{align*}
[\theta_1,\theta_2,\theta_3](z)
&=
-a^3,\\
[\theta_{12},\theta_3](z)
&=
(c_{12,1}+c_{12,2})a^3
=
(1-c_{12,3})a^3,\\
[\theta_{13},\theta_2](z)
&=
(c_{13,1}+c_{13,3})a^3
=
(1-c_{13,2})a^3,\\
[\theta_1,\theta_{23}](z)
&=
(c_{23,2}+c_{23,3})a^3
=
(1-c_{23,1})a^3.
\end{align*}
Thus
\begin{align*}
[\theta_1,\theta_2,\theta_3]_{\Theta_1}(z)
&=
-a^3
+
(1-c_{12,3})a^3
+
(1-c_{13,2})a^3
+
(1-c_{23,1})a^3\\
&=
\bigl(
2-c_{12,3}-c_{13,2}-c_{23,1}
\bigr)a^3.
\end{align*}
For any $\lambda\in\mathbb{Q}$, the parameters can be chosen so that
$
2-c_{12,3}-c_{13,2}-c_{23,1}=\lambda.
$
Hence the associated bracket can be chosen so that
\[
[\theta_1,\theta_2,\theta_3]_{\Theta_1}(z)
=
\lambda a^3.
\]

For $\varepsilon\in\{0,1\}$ and $\lambda\in\mathbb{Q}$, let
$
\zeta_\lambda^\varepsilon
\in
\Der^{-5}(\wedge V,A;\phi_\varepsilon)
$
be the derivation determined by
$
\zeta_\lambda^\varepsilon(x_i)=0
$
and
$
\zeta_\lambda^\varepsilon(z)=\lambda a^3.
$
These derivations are cycles. Moreover, every derivation of degree $-5$ is
of this form. 
Since
$
\Der^{-6}(\wedge V,A;\phi_\varepsilon)=0
$
for degree reasons, the cycles
$\zeta_\lambda^\varepsilon$ and
$\zeta_{\lambda'}^\varepsilon$ represent distinct cohomology classes whenever
$\lambda\neq\lambda'$.

Let
$
\alpha_i^\varepsilon
\in
\pi_2(\map(X,Y;f_\varepsilon))\otimes\mathbb{Q}
$
be the element corresponding to the cohomology class represented by
$\theta_i^\varepsilon$, and let
\[
\Phi_\varepsilon:
\pi_*(\map(X,Y;f_\varepsilon))\otimes\mathbb{Q}
\longrightarrow
H^{-*}(\Der(\wedge V,A;\phi_\varepsilon))
\]
denote the standard identification.
For $\varepsilon=0$, the second-order brackets vanish as derivations.
For $\varepsilon=1$, they are non-zero but exact. Hence both third-order
Whitehead products
\[
[\alpha_1^\varepsilon,
 \alpha_2^\varepsilon,
 \alpha_3^\varepsilon]_{\mathrm{Wh}}
\]
are non-empty.
By Theorem~\ref{thm:main}, we obtain the following.

\begin{enumerate}
\item In the component of $f_0$,
$
\Phi_0
\bigl(
[\alpha_1^0,\alpha_2^0,\alpha_3^0]_{\mathrm{Wh}}
\bigr)
=
\left\{
\zeta_{-1}^0
\right\}.
$

\item In the component of $f_1$,
$
\Phi_1
\bigl(
[\alpha_1^1,\alpha_2^1,\alpha_3^1]_{\mathrm{Wh}}
\bigr)
=
\left\{
\zeta_\lambda^1
\ \middle|\
\lambda\in\mathbb{Q}
\right\}.
$
\end{enumerate}

Thus the defining systems depend on the component of the mapping space.
In the component of the constant map, a trivial defining system is available
and the third-order product has a single non-zero value. 
In the component of $f_1$, the higher components of any defining system must
be non-zero, and the product contains both zero and non-zero elements.
\end{ex}

The previous example comes from a cubic differential. The next example shows
a different behavior: the Sullivan model has only quadratic differentials, so
the direct third-order bracket vanishes. Nevertheless, suitable choices of the
higher components of a defining system give a non-zero associated third-order
bracket.

\begin{ex}\label{ex:quadratic-cp2}
Let $X=\mathbb{C}P^2$ and put
$
A=H^*(\mathbb{C}P^2;\mathbb{Q})=\mathbb{Q}[a]/(a^3)
$
with $|a|=2$.
Regard $A$ as a CDGA with zero differential. Let $Y$ be a rational space
with minimal Sullivan model
\[
        (\wedge V,d)
        =
        (\wedge(x_1,x_2,x_3,z_{12},z_{13},z_{23}),d),
\]
where $|x_i|=4$, $|z_{ij}|=7$, $dx_i=0$, and $dz_{ij}=x_ix_j$
for $1\leq i<j\leq 3$.
Define a CDGA morphism
$
\phi:\wedge V\to A
$
by
$\phi(x_i)=a^2$ and $\phi(z_{ij})=0$.
For $i=1,2,3$, define
$
\theta_i\in \Der^{-2}(\wedge V,A;\phi)
$
by
\[
        \theta_i(x_j)=
        \begin{cases}
        a & (i=j),\\
        0 & (i\neq j),
        \end{cases}
        \qquad
        \theta_i(z_{kl})=0.
\]
Then each $\theta_i$ is a cycle.
The second-order brackets are as follows. For $1\leq i<j\leq 3$,
\[
        [\theta_i,\theta_j](z_{ij})=-a^2,
        \qquad
        [\theta_i,\theta_j](z_{kl})=0
        \quad
        \text{if } \{k,l\}\neq \{i,j\}.
\]
These brackets are non-zero as derivations, but they are exact.
Define derivations $\eta_{ij}\in \Der^{-4}(\wedge V,A;\phi)$ by
\[
\begin{array}{lll}
        \eta_{12}(x_1)=\frac{1}{2},&
        \eta_{12}(x_2)=\frac{1}{2},&
        \eta_{12}(x_3)=-\frac{1}{2},\\[2mm]
        \eta_{13}(x_1)=\frac{1}{2},&
        \eta_{13}(x_2)=-\frac{1}{2},&
        \eta_{13}(x_3)=\frac{1}{2},\\[2mm]
        \eta_{23}(x_1)=-\frac{1}{2},&
        \eta_{23}(x_2)=\frac{1}{2},&
        \eta_{23}(x_3)=\frac{1}{2},
\end{array}
\]
and set $\eta_{ij}(z_{kl})=0$ for all $1\leq k<l\leq 3$.
For example,
\[
\begin{aligned}
        (\delta\eta_{12})(z_{12})
        &=
        -\eta_{12}(x_1x_2)
        =
        -\left(\frac{1}{2}+\frac{1}{2}\right)a^2
        =
        -a^2,\\
        (\delta\eta_{12})(z_{13})
        &=
        -\eta_{12}(x_1x_3)
        =
        -\left(\frac{1}{2}-\frac{1}{2}\right)a^2
        =
        0,\\
        (\delta\eta_{12})(z_{23})
        &=
        -\eta_{12}(x_2x_3)
        =
        -\left(\frac{1}{2}-\frac{1}{2}\right)a^2
        =
        0.
\end{aligned}
\]
The same calculation for $\eta_{13}$ and $\eta_{23}$ gives
$
\delta\eta_{ij}=[\theta_i,\theta_j].
$
Put
$
\theta_{ij}=-\eta_{ij}.
$
Then
$
\delta\theta_{ij}=-[\theta_i,\theta_j].
$
Hence
\[
        \Theta=
        \{\theta_1,\theta_2,\theta_3,\theta_{12},\theta_{13},\theta_{23}\}
\]
is a defining system for $\theta_1,\theta_2,\theta_3$.

The associated bracket is computed as follows. Since all non-zero differentials
in the model of $Y$ are quadratic, the direct term
$
[\theta_1,\theta_2,\theta_3]
$
is zero. Hence
\[
        [\theta_1,\theta_2,\theta_3]_\Theta
        =
        [\theta_{12},\theta_3]
        +
        [\theta_{13},\theta_2]
        +
        [\theta_1,\theta_{23}].
\]
For instance,
\[
[\theta_{12},\theta_3](z_{13})
=
-\langle\theta_{12},\theta_3\rangle_\phi(x_1x_3)
=
\frac{1}{2}a.
\]
The other values are
\[
\begin{array}{lll}
        [\theta_{12},\theta_3](z_{12})=0,&
        [\theta_{12},\theta_3](z_{13})=\frac{1}{2}a,&
        [\theta_{12},\theta_3](z_{23})=\frac{1}{2}a,\\[2mm]
        [\theta_{13},\theta_2](z_{12})=\frac{1}{2}a,&
        [\theta_{13},\theta_2](z_{13})=0,&
        [\theta_{13},\theta_2](z_{23})=\frac{1}{2}a,\\[2mm]
        [\theta_1,\theta_{23}](z_{12})=\frac{1}{2}a,&
        [\theta_1,\theta_{23}](z_{13})=\frac{1}{2}a,&
        [\theta_1,\theta_{23}](z_{23})=0.
\end{array}
\]
Therefore
\[
        [\theta_1,\theta_2,\theta_3]_\Theta(z_{12})
        =
        [\theta_1,\theta_2,\theta_3]_\Theta(z_{13})
        =
        [\theta_1,\theta_2,\theta_3]_\Theta(z_{23})
        =
        a.
\]
Thus the associated bracket is non-zero.

The resulting class is also non-zero in cohomology. Let
$
\xi\in \Der^{-6}(\wedge V,A;\phi).
$
Then, for $1\leq i<j\leq 3$, $(\delta\xi)(z_{ij}) = 0$ for degree reasons.
Hence no boundary in degree $-5$ can take the value $a$ on any $z_{ij}$. Thus
$
[\theta_1,\theta_2,\theta_3]_\Theta
$
represents a non-zero cohomology class in
$
H^{-5}(\Der(\wedge V,A;\phi)).
$

Let $f:X\to Y$ be the map represented by $\phi$, and let
$\alpha_i\in \pi_2(\map(X,Y;f))\otimes\mathbb{Q}$
be the element corresponding to the cohomology class of $\theta_i$.
Since the second-order brackets $[\theta_i,\theta_j]$ are exact, the corresponding second-order Whitehead products vanish rationally. 
Hence the third-order Whitehead product
$
        [\alpha_1,\alpha_2,\alpha_3]_{\mathrm{Wh}}
$
is non-empty. By Theorem~\ref{thm:main}, its image under the identification
\eqref{eq:adjoint-derivation-isomorphism} contains the non-zero class
represented by $[\theta_1,\theta_2,\theta_3]_\Theta$. In particular,
$[\alpha_1,\alpha_2,\alpha_3]_{\mathrm{Wh}}$ contains a non-zero element in
$
        \pi_5(\map(\mathbb{C}P^2,Y;f))\otimes\mathbb{Q}.
$
\end{ex}

\section*{Acknowledgments}

The author is grateful to Katsuhiko Kuribayashi, Shun Wakatsuki and Toshihiro Yamaguchi for their helpful comments and suggestions on an earlier version of this manuscript.
This work was partially supported by JSPS KAKENHI Grant Number JP26K06813.

\end{document}